\documentclass[11pt]{amsart}
\usepackage {enumerate}
\usepackage{setspace}
\usepackage{caption}
\usepackage{mathrsfs}
\usepackage{hyperref}
\usepackage{esint}
\usepackage{amssymb}
\usepackage{graphicx}
\usepackage{epstopdf}
\usepackage{amsmath}

\usepackage{color}
\usepackage{lipsum}

\newtheorem{theorem}{Theorem}[section]
\newtheorem{lemma}[theorem]{Lemma}
\newtheorem{corollary}[theorem]{Corollary}
\newtheorem{proposition}[theorem]{Proposition}

\newtheorem{conjecture}[theorem]{Conjecture}
\numberwithin{equation}{section}

\theoremstyle {definition}
\newtheorem{definition}[theorem]{Definition}
\newtheorem{remark}[theorem]{Remark}

\DeclareMathOperator{\Ric}{Ric}

\DeclareMathOperator{\image}{Im}

\makeatletter
\@namedef{subjclassname@2020}{%
  \textup{2020} Mathematics Subject Classification}
\makeatother

\begin{document}
\title[PMT with arbitrary ends]{Positive mass theorem with arbitrary ends and its application}
\author{Jintian Zhu}
\address{Beijing International Center for Mathematical Research, Peking University, Beijing, 100871, P.~R.~China}
\email{zhujintian@bicmr.pku.edu.cn}
%\date{\today}

\begin{abstract}
In this article, we give a proof for positive mass theorem of asymptotically flat manifolds with arbitrary ends when the dimension is no greater than seven. As an application, we also show a positive mass theorem for asymptotically locally Euclidean manifolds with necessary incompressible conditions.
\end{abstract}
\maketitle
%\blfootnote{2020 {\it Mathematics Subject Classification}. Primary 53C21; Secondary 83C99}

\section{Introduction}

In mathematical general relativity, one of the most beautiful results is the {\it positive mass theorem} proved by Schoen and Yau, which states that every complete {\it asymptotically flat} $3$-manifold with nonnegative scalar curvature has nonnegative ADM mass and the mass vanishes exactly when it is the Euclidean $3$-space. Shortly after that, Schoen \cite{Schoen1989} was able to generalize this result to higher dimensions no greater than seven based on dimension descent argument. Under additional spin condition, Witten \cite{Witten81} gave a proof for the positive mass theorem in all dimensions with Dirac operator method. Among all these results, asymptotically flat manifolds under consideration are assumed to have all its ends close to the Euclidean space.

In their book
\cite{SY94}, Schoen and Yau made the conjecture that positive mass theorem still holds for asymptotically Schwarzschild manifolds with nonnegative scalar curvature even if some of its ends are complete but far from the Euclidean space. Such manifolds can be referred to as {\it asymptotically flat manifolds with arbitrary ends} and the precise meaning is given in the definition below. Let us assume that $n$ is an integer no less than three throughout the paper.

\begin{definition}\label{Def 1}
A smooth {\it complete} Riemannian manifold $(M,g)$ with {\it no boundary} and a distinguished end $\mathcal E$ is called an asymptotically flat manifold with arbitrary ends if
\begin{itemize}
\item $\mathcal E$ is diffeomorphic to $\mathbb R^n-\bar B_1$;
\item the metric $g$ restricted to $\mathcal E$ has the expression
$
g_{ij}=\delta_{ij}+h_{ij},
$
where the error term $h$ satisfies the decay condition
\begin{equation}\label{Eq: decay condition}
|h|+r|\partial h|+r^2|\partial\partial h|\leq Cr^{2-n},\quad r=|x|.
\end{equation}
\item the scalar curvature satisfies $R(g)\leq Cr^{-q}$ in $\mathcal E$ for some $q>n$.
\end{itemize}
\end{definition}

The phrase ``arbitrary ends'' here means that one does not need to impose additional requirements except completeness on those ends other than $\mathcal E$. The figure in next page illustrates the difference between our Definition \ref{Def 1} and the classical one.
\begin{figure}[htbp]\label{Fig: 1}
\centering
\includegraphics[width=11cm]{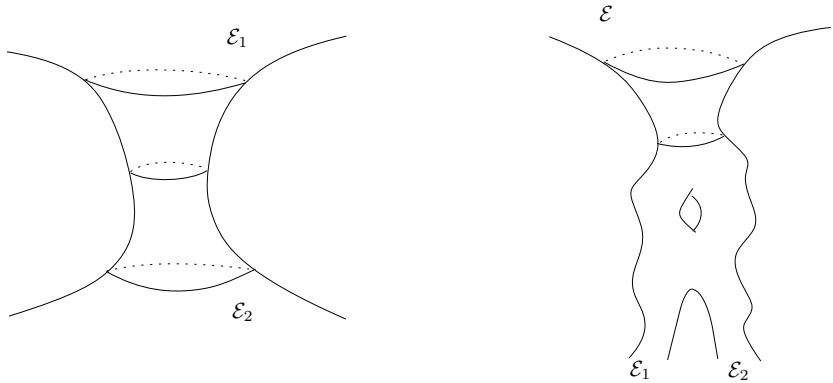}
\caption{\small The left picture presents a classical asymptotically flat manifold with two ends $\mathcal E_1$ and $\mathcal E_2$, both of which are asymptotic to the Euclidean space. The right one shows an asymptotically flat manifold with arbitrary ends, where $\mathcal E$ is  the distinguished end asymptotic to the Euclidean space while $\mathcal E_1$ and $\mathcal E_2$ are only complete with no specific asymptotics.}
\end{figure}

The original motivation of Schoen-Yau's conjecture is to prove the so-called Liouville theorem in conformal geometry: if a complete and locally conformally flat Riemannian $n$-manifold $M$ with nonnegative scalar curvature has a conformal map $\Phi$ to $n$-sphere, then $\Phi$ is injective and $\partial \Phi(M)$ has zero Neumann capacity. Recently, Lesourd, Unger and Yau \cite{LUY2021} successfully proved the nonnegativity of ADM mass for asymptotically Schwarzschild manifolds with arbitrary ends as well as nonnegative scalar curvature up to dimension seven. In their proof, they established an adapted version of Schoen's dimension descent argument, where minimal hypersurfaces in the original proof are replaced by soap bubbles. It turns out that Lesourd-Unger-Yau's result is enough to complete above program of Schoen and Yau.

Back to the positive mass theorem for asymptotically flat manifolds with arbitrary ends, the answer given by Lesourd, Unger and Yau can be further improved since the rigidity part was not dealt with there and asymptotically Schwarzschild could be generalized to more general asymptotics. A complete answer turns out to be desirable in order to obtain right comprehension on effects of arbitrary ends. On the other hand, Lesourd, Unger and Yau expressed their concern in \cite{LUY2021} on whether the usual procedure --- the density theorem in \cite{SY81} as well as Lohkamp's compactification in \cite{Lohkamp99} can be modified in the presence of arbitrary ends. At the first sight, these ends may lead to a lack of controls of involved PDEs.

In this paper, we would like to prove the positive mass theorem for asymptotically flat manifolds with arbitrary ends based on the idea of density theorem and Lohkamp's compactification after introducing some technical modifications. The key observation here is that only positive lower bound are required on arbitrary ends in the construction of various conformal factors. In order to construct such conformal factors from solving elliptic equations (see Proposition \ref{Prop: the conformal factor}), the Harnack inequality is enough to guarantee their positivity on arbitrary ends. In our use of these conformal factors in Section \ref{Sec 3}, a positive lower bound will be achieved just by adding a small positive constant and doing renormalization.

In the following, $(M,g,\mathcal E)$ is always denoted to be an asymptotically flat manifold with arbitrary ends, where $\mathcal E$ is the distinguished asymptotically flat end.
The {\it ADM mass} of $(M,g,\mathcal E)$ is defined to be that of the  end $\mathcal E$ given by the limit
$$
\lim_{\rho\to\infty}\frac{1}{2(n-1)|\mathbb S^{n-1}|}\int_{\{r=\rho\}}(g_{ij,j}-g_{jj,i})\,\mathrm d\sigma^i,
$$
where $g_{ij,k}$ is denoted to be the partial derivative $\partial_k g_{ij}$ and $\mathrm d\sigma^i$ is the normal area element of $\{r=\rho\}$ with respect to the Euclidean metric. A famous result of Bartnik in \cite{Bartnik86} (see also Chru\'sciel's work \cite{Chr1986}) shows that the ADM mass is a geometric quantity independent of the choice of coordinate system at infinity.

Now we can state our main result as
\begin{theorem}\label{Thm: main 1}
Let $n\leq 7$. If $(M,g,\mathcal E)$ has nonnegative scalar curvature, then its ADM mass is nonnegative. Moreover, the ADM mass vanishes if and only if $(M,g)$ is isometric to the Euclidean space.
\end{theorem}
\begin{remark}
The completeness from Definition \ref{Def 1} for asymptotically flat manifolds with arbitrary ends is crucial for Theorem \ref{Thm: main 1}. Otherwise, it is easy to see that Schwarzschild manifolds with negative mass serve as counterexamples.
\end{remark}

Roughly speaking, Theorem \ref{Thm: main 1} indicates that complete ends can not lead to a large mass drop on the asymptotically flat end and it is interesting to figure out whether the same phenomenon happens in the setting of Penrose inequality.

Closely related to our Theorem \ref{Thm: main 1}, we mention the {\it Geroch conjecture with arbitrary ends} below.
\begin{conjecture}\label{Conj: Geroch}
For any $n$-manifold $X$, there is no smooth complete metric on $T^n\sharp X$ with positive scalar curvature. Moreover, the only complete metric on $T^n\sharp X$ with nonnegative scalar curvature is flat.
\end{conjecture}
With soap bubble method, Chodosh and Li \cite{CL2020} gave an affirmative answer to above conjecture with dimension no greater than seven.
It is well-known to experts that this dimension restriction comes from possible singularity issues for minimizing soap bubbles. However, this difficulty may be overcome with a similar method from the work \cite{SY17}, where Schoen and Yau solved classical Geroch conjecture in any dimension based on a subtle analysis on singular minimal slicings. This suggests that Conjecture \ref{Conj: Geroch} may hold in all dimensions. Similar to the proof for classical positive mass theorem, we are going to reduce our Theorem \ref{Thm: main 1} to Conjecture \ref{Conj: Geroch}, which is already a theorem in dimensions no greater than seven.

As an application of Theorem \ref{Thm: main 1}, we can obtain a positive mass theorem for {\it asymptotically locally Euclidean} (ALE) manifolds. Recall
\begin{definition}
Given any finite group $\Gamma\subset O(n)$ acting freely on $\mathbb R^n-\{0\}$, a complete Riemannian $n$-manifold $(M_\Gamma,g_\Gamma,\mathcal E_\Gamma)$ without boundary is ALE of group $\Gamma$ with arbitrary ends if
\begin{itemize}
\item there is a diffeomorphism $\Phi:\mathcal E_\Gamma\to (\mathbb R^n-\bar B_1)/\Gamma$ such that
$$
(\pi\circ \Phi^{-1})^*(g_\Gamma)=\delta_{ij}+h_{ij}\quad\text{in}\quad \mathbb R^n-B_1
$$
with
$$
|h|+r|\partial h|+r^2|\partial\partial h|\leq Cr^{2-n},\quad r=|x|,
$$
where $\pi:(\mathbb R^n-B_1)\to (\mathbb R^n-B_1)/\Gamma$ is the canonical projection map.
\item the scalar curvature of the pullback metric $(\pi\circ \Phi^{-1})^*(g_\Gamma)$ is no greater than $Cr^{-q}$ in $\mathbb R^n-B_1$ for some $q>n$.
\end{itemize}
\end{definition}
\begin{definition}
Given any ALE manifold $(M_\Gamma,g_\Gamma,\mathcal E_\Gamma)$ of group $\Gamma$ with arbitrary ends, the ADM mass is defined to be
$$
\lim_{\rho\to\infty}\frac{1}{2(n-1)|\Gamma||\mathbb S^{n-1}|}\int_{\{r=\rho\}}(g_{ij,j}-g_{jj,i})\,\mathrm d\sigma^i,\quad g=(\pi\circ \Phi^{-1})^*(g_\Gamma).
$$
\end{definition}

In their work \cite{HP1978}, Hawking and Pope came up with the {\it general positive energy conjecture}: the ADM mass is nonnegative for all ALE $4$-manifolds with vanishing scalar curvature. However, a family of counterexamples to this conjecture were soon constructed by LeBrun in \cite{LeBrun1988}. We point out that those counterexamples are simply connected while their end has fundamental group isometric to $\mathbb Z_k$ with $k\geq 3$. The compressibility of the end turns out to be the key reason why positive mass theorem fails from the view of the following corollary.
\begin{corollary}\label{Cor: ALE}
Let $n\leq 7$. Denote $(M_\Gamma,g_{\Gamma},\mathcal E_\Gamma)$ to be an ALE manifold of non-trivial finite group $\Gamma$ with arbitrary ends such that the inclusion map $i_*:\pi_1(\mathcal E_\Gamma)\to \pi_1(M_{\Gamma})$ is injective. If it has non-negative scalar curvature, then it has positive ADM mass.
\end{corollary}
Despite a simple consequence of Theorem \ref{Thm: main 1}, Corollary \ref{Cor: ALE} belongs to the class of {\it positive mass theorems with incompressible condition}. The use of incompressible condition dates back to the works \cite{SY1979} and \cite{SY1982}, where Schoen and Yau proved that a complete Riemannian $3$-manifold $(M,g)$ cannot admit an immersed $2$-torus $\Sigma$ with the injective inclusion map $i_*:\pi_1(\Sigma)\to \pi_1(M)$ unless $(M,g)$ is isometrically covered by the Riemannian product of a flat $2$-torus and the real line. In recent work \cite{LSZ2021}, the author and his collaborators proved a positive mass theorem for a special class of ALF and ALG manifolds $(\tilde M,\tilde g,\tilde {\mathcal E})$ under the assumption that the inclusion map $i_*:\pi_1(\tilde{\mathcal E})\to \pi_1(\tilde M)$ is injective (and only $\image i_*\neq 0$ is required in ALF case).  Given these results, it is interesting to figure out whether positive mass theorem holds for manifolds with general asymptotics when an additional incompressible condition is imposed.

The rest part of the paper will be organized as follows. First we establish necessary preliminary analytic results in Section 2 and then we give detailed proofs for Theorem \ref{Thm: main 1} and Corollary \ref{Cor: ALE} in Section 3. 
\newline\\
{\bf Acknowledgments.} The author would like to thank professor Yuguang Shi for many inspiring conversations. This research is supported by the China post-doctoral grant BX2021013.

\section{Preparation in analysis}
The following lemma comes from \cite{SY79}.
\begin{lemma}\label{Lem: Sobolev}
Let $U$ be a smooth open neighborhood of $\mathcal E$ such that $U-\mathcal E$ has compact closure in $M$. There is a positive constant $c_S$ depending on $U$ and $g$ such that
\begin{equation}\label{Eq: Sobolev}
c_S\left(\int_U|\zeta|^\frac{2n}{n-2}\,\mathrm d\mu_g\right)^{\frac{n-2}{n}}\leq \int_U|\nabla\zeta|^2\,\mathrm d\mu_g
\end{equation}
for any smooth function with compact support in $\bar U$.
\end{lemma}
\begin{proposition}\label{Prop: the conformal factor}
Let $U$ and $c_S$ be the same as in Lemma \ref{Lem: Sobolev}. Assume that $f$ is a smooth function on $M$ with compact support in $U$ such that
\begin{itemize}
\item the negative part $f_-$ of $f$ satisfies
\begin{equation}\label{Eq: f slight negative}
\left(\int_{U}|f_-|^{\frac{n}{2}}\,\mathrm d\mu_g\right)^{\frac{2}{n}}\leq \frac{c_S}{2};
\end{equation}
%\item for some constants $C_0$ there holds
%\begin{equation}\label{Eq: f integral}
%\left(\int_{U}|f|^{\frac{2n}{n+2}}\,\mathrm d\mu_g\right)^{\frac{n+2}{2n}}\leq C_0.
%\end{equation}
\end{itemize}
Then the equation
\begin{equation*}
\Delta_g u-fu=0
\end{equation*}
has a positive solution $u$ on $M$ such that $u$ has the expansion
\begin{equation*}
u=1+Ar^{2-n}+\omega
\end{equation*}
with
$$
A=-\frac{1}{(n-2)|\mathbb S^{n-1}|}\int_U fu\,\mathrm d\mu_g
$$
and
$
|\omega|+r|\partial \omega|+r^2|\partial^2 \omega|\leq Cr^{1-n}
$
in $\mathcal E$. Moreover, we have
$$
\int_{\partial U}\frac{\partial u}{\partial \vec n}\,\mathrm d\sigma_g=0\quad\text{and}\quad \int_{\partial U}u\frac{\partial u}{\partial\vec n}\,\mathrm d\sigma_g\leq 0,
$$
where $\vec n$ is the outward unit normal of $\partial U$.
\end{proposition}
\begin{remark}
Compared to \cite[Lemma 3.2]{SY79}, the role of inner boundaries is now replaced by the infinity of arbitrary ends and so no boundary condition can be imposed.
\end{remark}
\begin{proof}
First let us take a smooth exhaustion
$$
U=U_0\subset U_1\subset U_2\subset\cdots
$$
such that $U_i-\mathcal E$ has a compact closure for all $i$ and
$$
M=\bigcup_{i=0}^{\infty}U_i.
$$
Such exhaustion can be constructed from collecting points no greater than a certain distance to $U$. Fix some $U_i$ and consider the following equation
\begin{equation}\label{Eq: v_i}
\Delta_g v_i-fv_i=f\ \ \text{in}\ \,U_i;\quad \frac{\partial v_i}{\partial \vec n}=0\ \ \text{on}\ \,\partial U_i,
\end{equation}
where $\vec n$ is denoted to be the outward unit normal of $\partial U_i$. It is standard that equation \eqref{Eq: v_i} can be solved by exhaustion method, which we review now. For any $R>1$, we take
$$
U_{i,R}=\{x\in U_i:r(x)\leq R\}
$$
and then $\partial U_{i,R}=\partial U_i\cup\partial B_R$, where
$\partial B_R=\{x\in M:r(x)=R\}\subset \mathcal E.$
With this setting, we consider the following equation
\begin{equation}\label{Eq: U_iR}
\left\{
\begin{array}{ccc}
\Delta_g v_{i,R}-fv_{i,R}=f&\text{in}& U_{i,R},\\
\frac{\partial v_{i,R}}{\partial \vec n}=0&\text{on}&\partial U_i,\\
v_{i,R}=0&\text{on}&\partial B_R.
\end{array}
\right.
\end{equation}
From Fredholm alternative \cite[Theorem 5.3]{GT2001}, equation \eqref{Eq: U_iR} has a smooth solution $v_{i,R}$ once the corresponding homogeneous equation to \eqref{Eq: U_iR} has only zero solution and so we verify this briefly. Let $\zeta$ solve the homogeneous equation. From integration by parts we see
\begin{equation*}
\int_{U_{i,R}}|\nabla_g\zeta|^2\,\mathrm d\mu_g=-\int_{U_{i,R}}f\zeta^2\,\mathrm d\mu_g.
\end{equation*}
Denote $U_R=\{x\in U:r(x)\leq R\}$. From \eqref{Eq: Sobolev} the left hand side is no less than
$$
c_S\left(\int_{U_R}|\zeta|^{\frac{2n}{n-2}}\,\mathrm d\mu_g\right)^{\frac{n-2}{n}}.
$$
However, the right hand side is no greater than
$$
\left(\int_{U_R}|f_-|^{\frac{n}{2}}\,\mathrm d\mu_g\right)^{\frac{2}{n}}\left(\int_{U_R}|\zeta|^{\frac{2n}{n-2}}\,\mathrm d\mu_g\right)^{\frac{n-2}{n}},
$$
where we use the H\"older inequality and the fact that $f$ has compact support in $U$. Combined with \eqref{Eq: f slight negative} we obtain
$$
\int_{U_R}|\zeta|^{\frac{2n}{n-2}}\,\mathrm d\mu_g=0
$$
and so $\zeta$ vanishes in $U_R$. Notice that $\zeta$ is a harmonic function in $U_{i,R}-U_R$. Then the maximum principle yields $\zeta\equiv 0$ and a solution $v_{i,R}$ of \eqref{Eq: U_iR} follows. From integration by parts to $v_{i,R}$ and H\"older inequality, we also have
\begin{equation*}
\begin{split}
\int_{U_{i,R}}|\nabla_g v_{i,R}|^2\,\mathrm d\mu_g\leq &\left(\int_{U_R}|f_-|^{\frac{n}{2}}\,\mathrm d\mu_g\right)^{\frac{2}{n}}\left(\int_{U_R}|v_{i,R}|^{\frac{2n}{n-2}}\,\mathrm d\mu_g\right)^{\frac{n-2}{n}}\\
&\qquad+\left(\int_{U_R}|f|^{\frac{2n}{n+2}}\,\mathrm d\mu_g\right)^{\frac{n+2}{2n}}\left(\int_{U_R}|v_{i,R}|^{\frac{2n}{n-2}}\,\mathrm d\mu_g\right)^{\frac{n-2}{2n}}.
\end{split}
\end{equation*}
Combined with \eqref{Eq: Sobolev} and \eqref{Eq: f slight negative}, it follows
$$
\left(\int_{U_{R}}|v_{i,R}|^{\frac{2n}{n-2}}\,\mathrm d\mu_g\right)^{\frac{n-2}{2n}}\leq 2c_S^{-1}C_0\quad\mbox{with}\quad C_0:=\left(\int_{U}|f|^{\frac{2n}{n+2}}\,\mathrm d\mu_g\right)^{\frac{n+2}{2n}}.
$$
Take an increasing sequence $R_j$ such that $R_j\to+\infty$ as $j\to\infty$. From the $L^p$-estimate \cite[Theorem 9.11]{GT2001} and Schauder estimate \cite[Theorem 6.2]{GT2001}, we have
\begin{equation}\label{Eq: local uniform bound}
\sup_K|\nabla_g^kv_{i,R_j}|\leq C(k,K)
\end{equation}
for any compact subset $K$ of $U$, where $C(k,K)$ is a constant depending only on $k$ and $K$. Fix a positive constant $r_0>1$. Now we show that $v_{i,R_j}$ in $U_{i,r_0}$ is uniformly bounded by a constant $C(i)$ independent of $j$. Otherwise, there is a subsequence (still denoted by $v_{i,R_j}$) such that
$$
\sup_{U_{i,r_0}}|v_{i,R_j}|\to +\infty\quad\text{but}\quad\left(\int_{U_{r_0}}|v_{i,R_j}|^{\frac{2n}{n-2}}\,\mathrm d\mu_g\right)^{\frac{n-2}{2n}}\leq 2c_S^{-1}C_0.
$$
Let
$$
w_{i,R_j}=\left(\sup_{U_{i,r_0}}|v_{i,R_j}|\right)^{-1}v_{i,R_j}.
$$
Clearly $|w_{i,R_j}|\leq 1$ in $U_{i,r_0}$ and so it follows from the Schauder estimates \cite[Theorem 6.2 and Theorem 6.30]{GT2001} that up to a subsequence $w_{i,R_j}$ converges uniformly to a limit function $w_i$ solving
$$\Delta_g w_i-f w_i=0\quad \text{in}\quad U_{i,r_0};\quad \frac{\partial w_i}{\partial\vec n}=0\quad \text{on}\quad \partial U_i,$$
and satisfying
$$
\sup_{U_{i,r_0}}|w_i|=1\quad\text{and}\quad w_i= 0\quad \text{in}\quad U_{r_0}.
$$
As before, $w_i$ is a harmonic function in $U_{i,r_0}-U_{r_0}$ and then the maximum principle yields $w_i\equiv 0$ in $U_{i,r_0}$, which leads to a contradiction. Now the uniform bound for $v_{i,R_j}$ in $U_{i,r_0}$ combined with Schauder estimate as well as estimate \eqref{Eq: local uniform bound} implies $v_{i,R_j}$ converges to a smooth solution $v_i$ of \eqref{Eq: v_i} up to a subsequence. Clearly we have
$$
\left(\int_{U}|v_i|^{\frac{2n}{n-2}}\,\mathrm d\mu_g\right)^{\frac{n-2}{2n}}\leq 2c_S^{-1}C_0.
$$
From the $L^\infty$-bound of $f$ and the asymptotically flatness of $\mathcal E$ we can apply the $L^p$-estimate to conclude
\begin{equation}\label{Eq: L infty bound vi}
\sup_{U}|v_i|\leq C_1
\end{equation}
for some constant $C_1$ independent of $i$ and also
\begin{equation}\label{Eq: 0 limit vi}
\lim_{r(x)\to +\infty}v_i(x)=0.
\end{equation}
Notice that $v_i$ is a harmonic function in $U_i-U$. Then the divergence theorem yields
\begin{equation}\label{Eq: v_i p1}
\int_{\partial U}\frac{\partial v_i}{\partial \vec n}\,\mathrm d\sigma_g=\int_{\partial U_i}\frac{\partial v_i}{\partial\vec n}\,\mathrm d\sigma_g=0
\end{equation}
and
\begin{equation}\label{Eq: v_i p2}
\int_{\partial U}v_i\frac{\partial v_i}{\partial \vec n}\,\mathrm d\sigma_g=-\int_{U_i-U}|\nabla_g v_i|^2\,\mathrm d\sigma_g\leq 0.
\end{equation}

Next we consider function $u_i=1+v_i$ instead. Clearly, $u_i$ solves
\begin{equation*}
\Delta_g u_i-fu_i=0\ \ \text{in}\ \,U_i;\quad \frac{\partial u_i}{\partial \vec n}=0\ \ \text{on}\ \,\partial U_i,
\end{equation*}
We claim that $u_i$ is positive everywhere in $U_i$. Observe that $u_i$ tends to one at the infinity of $\mathcal E$ due to \eqref{Eq: 0 limit vi}. All we need to show is the nonnegativity of $u_i$ and the positivity will come from the Harnack inequality \cite[Theorem 8.20]{GT2001}. Suppose that the set $\Omega_{i,-}=\{u_i<0\}$ is non-empty. Since $u_i$ is positive at the infinity of $\mathcal E$, the set $\Omega_{i,-}$ is compact. From integration by parts as well as the boundary condition of $u_i$, we have
\begin{equation*}
\int_{\Omega_{i,-}}|\nabla_g u_i|^2\,\mathrm d\mu_g=-\int_{\Omega_{i,-}}fu_i^2\,\mathrm d\mu_g.
\end{equation*}
As before, we conclude from \eqref{Eq: Sobolev} and \eqref{Eq: f slight negative} as well as the maximum principle that $u_i$ vanishes in $\Omega_{i,-}$.
%which is obviously impossible.
Using the strong unique continuation property from \cite{GL1987}, we see that $u_i$ is identical to zero, which is impossible.
Now, we are able to apply the Harnack inequality \cite[Theorem 8.20]{GT2001} to obtain local smooth convergence of $u_i$. In fact, functions $u_i$ have a uniform bound in any compact subset of $M$ from \eqref{Eq: L infty bound vi}. Combined with Schauder estimates, it implies that $u_i$ converges smoothly to a nonnegative limit function $u$ up to a subsequence. Clearly $u$ solves the equation
$
\Delta_g u-f u=0
$
in $M$. Denote $v=u-1$, then we have
\begin{equation}\label{Eq: integral v}
\left(\int_{U}|v|^{\frac{2n}{n-2}}\,\mathrm d\mu_g\right)^{\frac{n-2}{2n}}\leq 2c_S^{-1}C_0.
\end{equation}
From $L^p$-estimate it follows
\begin{equation}\label{Eq: estimate v}
\sup_{U}|v|\leq C_1\quad \text{and}\quad\lim_{r(x)\to +\infty}v(x)=0.
\end{equation}
In particular, $u$ tends to one at the infinity of $\mathcal E$ and the Harnack inequality \cite[Theorem 8.20]{GT2001} yields that it is positive everywhere. From the proof of \cite[Lemma 3.2]{SY79}, $v$ has the expansion
$$
v=\frac{A}{r^{n-2}}+\omega
$$
with
$$
A=-\frac{1}{(n-2)|\mathbb S^{n-1}|}\int_{U}f(v+1)\,\mathrm d\mu_g
$$
and
$
|\omega|+r|\partial\omega|+r^2|\partial^2\omega|\leq Cr^{1-n}
$ in $\mathcal E$. The last statement comes directly from \eqref{Eq: v_i p1} and \eqref{Eq: v_i p2} as well as the smooth convergence of $u_i$.
\end{proof}

\section{Proof for Theorem \ref{Thm: main 1} and Corollary \ref{Cor: ALE}}\label{Sec 3}
First we establish the following density theorem as in \cite{SY81} and the audience should pay attention to the different choice of conformal factors below compared to \cite{SY81}. As a preparation, let us recall the definition of {\it asymptotically Schwarzschild manifold with arbitrary ends}.
\begin{definition}
A smooth complete Riemannian manifold $(M,g)$ with no boundary and a distinguished end $\mathcal E$ is called an asymptotically Schwarzschild manifold with arbitrary ends if
\begin{itemize}
\item $\mathcal E$ is diffeomorphic to $\mathbb R^n-\bar B_1$;
\item the metric $g$ restricted to $\mathcal E$ has the expression
\begin{equation}\label{Eq: AS expression}
g_{ij}=\left(1+\frac{m}{2r^{n-2}}\right)^{\frac{4}{n-2}}\delta_{ij}+h_{ij},
\end{equation}
where the error term $h$ satisfies the decay condition
\begin{equation*}%\label{Eq: decay condition 2}
|h|+r|\partial h|+r^2|\partial\partial h|\leq Cr^{1-n},\quad r=|x|.
\end{equation*}
\end{itemize}
We point out that the ADM mass of an asymptotically Schwarzschild manifold $(M,g,\mathcal E)$ with arbitrary end equals to $m$ in the expression \eqref{Eq: AS expression}.
\end{definition}

The density theorem is stated as following
\begin{proposition}\label{Prop: deformation}
Assume that $(M,g,\mathcal E)$ has nonnegative scalar curvature whose ADM mass equals to $m$. For any $\epsilon>0$, we can construct a new metric $\bar g$ on $M$ such that
$(M,\bar g,\mathcal E)$ is an asymptotically Schwarzschild manifold with arbitrary ends, which has nonnegative scalar curvature and ADM mass $\bar m$ satisfying $|\bar m-m|\leq \epsilon$. Moreover, we can require $\|\bar g-g\|_g\leq \epsilon$ in the distinguished end $\mathcal E$.
\end{proposition}

\begin{proof}
First we can write the metric $g$ as
$$
g_{ij}=\left(1+\frac{m}{2r^{n-2}}\right)^{\frac{4}{n-2}}\delta_{ij}+\tilde g_{ij}.
$$
From the definition of ADM mass we have
\begin{equation}\label{Eq: no extra mass}
\lim_{\rho\to+\infty}\int_{S_\rho}(\tilde g_{ij,j}-\tilde g_{jj,i})\,\mathrm d\sigma^i=0.
\end{equation}
Take a fixed nonnegative cutoff function $\zeta:\mathbb R\to [0,1]$ such that $\zeta\equiv 0$ in $(-\infty,2]$, $\zeta\equiv 1$ in $[3,+\infty)$. With $s>1$ a constant to be determined later, we define
\begin{equation}\label{Eq: expression gs}
\hat g^s_{ij}
%=g_{ij}-\zeta\left(\frac{r}{s}\right)\tilde g_{ij}
=\left(1+\frac{m}{r^{n-2}}\right)^{\frac{4}{n-2}}\delta_{ij}+\left(1-\zeta\left(\frac{r}{s}\right)\right)\tilde g_{ij}.
\end{equation}
Clearly we have
\begin{equation*}
|\hat g_{ij}^s-\delta_{ij}|+r|\partial\hat g_{ij}^s|+r^2|\partial \partial\hat g_{ij}^s|\leq Cr^{2-n},
\end{equation*}
where $C$ is always denoted to be a universal constant independent of $s$ here and in the sequel. In particular, the metrics $\hat g^s$ are uniformly equivalent to $g$ in some neighborhood $U$ of $\mathcal E$, and so the Sobolev inequality \eqref{Eq: Sobolev} holds for all metrics $\hat g^s$ with a uniform constant $c_S$ independent of $s$. Through a straightforward computation, we can also see
$
R(\hat g^s)\geq 0
$
in $\{r\leq 2s\}$, $R(\hat g^s)\equiv 0$ in $\{r\geq 3s\}$ and
$
|R(\hat g^s)|\leq Cs^{-n}$ in $\{s\leq r\leq 4s\}$.
So we have
\begin{equation}\label{Eq: integral scalar}
\left(\int_{\{s\leq r\leq 4s\}}|R(\hat g^s)|^{\frac{2n}{n+2}}\,\mathrm d\mu_{\hat g^s}\right)^{\frac{n+2}{2n}}\leq Cs^{\frac{2-n}{2}}
\end{equation}
and
$$
\left(\int_U|R(\hat g^s)_-|^{\frac{n}{2}}\,\mathrm d\mu_{\hat g^s}\right)^{\frac{2}{n}}\leq Cs^{2-n}\leq \frac{c_S}{4}
$$
for sufficiently large $s$.

Now we take another nonnegative cutoff function $\eta:\mathbb R\to [0,1]$ such that $\eta\equiv 0$ in $(-\infty,1]\cup[4,+\infty)$ and $\eta\equiv 1$ in $[2,3]$. Let $\eta_s(x)=\eta(r(x)/s)$. For any $s$ we can take a positive constant $\delta_s$ such that
\begin{equation}\label{Eq: Lp}
\left(\int_U|\left(\eta_sR(\hat g^s)-\delta_s\eta_s\right)_-|^{\frac{n}{2}}\,\mathrm d\mu_{\hat g^s}\right)^{\frac{2}{n}}\leq \frac{c_S}{2}
\end{equation}
and
\begin{equation}\label{Eq: delta_s}
\delta_s \left(1+\mathcal H^n_{\hat g^s}(\{s\leq r\leq 4s\})\right)\leq s^{-1}.
\end{equation}
Such $\delta_s$ exists since the condition \eqref{Eq: Lp} only depends on the continuity of the $L^p$ norm.
From Proposition \ref{Prop: the conformal factor} we can construct a solution $u_s$ of the following equation
$$
\Delta_{\hat g^s} u_s-\frac{n-2}{4(n-1)}\left(\eta_s R(\hat g^s)-\delta_s\eta_s\right)u_s=0
$$
and $u_s$ has the expansion $u_s=1+A_s r^{2-n}+O(r^{1-n})$ with
$$
A_s=-\frac{1}{4(n-1)|\mathbb S^{n-1}|}\int_U \left(\eta_s R(\hat g^s)-\delta_s\eta_s\right)u_s\,\mathrm d\mu_{\hat g^s}.
$$
Denote $v_s=u_s-1$. It is clear that
\begin{equation}\label{Eq: 001}
\begin{split}
&\left|\int_U\eta_sR(\hat g^s) u_s\,\mathrm d\mu_{\hat g^s}\right|\\
\leq &\left(\int_{\{s\leq r\leq 4s\}}|R(\hat g^s)|^{\frac{2n}{n+2}}\,\mathrm d\mu_{\hat g^s}\right)^{\frac{n+2}{2n}}\left(\int_{\{s\leq r\leq 4s\}}|v_s|^{\frac{2n}{n-2}}\,\mathrm d\mu_{\hat g^s}\right)^{\frac{n-2}{2n}}\\
&\qquad\qquad+\left|\int_{\{s\leq r\leq 4s\}}\eta_sR(\hat g^s)\,\mathrm d\mu_{\hat g^s}\right|
\end{split}
\end{equation}
and
\begin{equation*}
\begin{split}
&\left|\int_U\delta_s\eta_s u_s\,\mathrm d\mu_{\hat g_s}\right|\leq  \delta_s\mathcal H^n_{\hat g^s}(\{s\leq r\leq 4s\})\\
+\delta_s &[\mathcal H^n_{\hat g^s}(\{s\leq r\leq 4s\})]^{\frac{n+2}{2n}}\left(\int_{\{s\leq r\leq 4s\}}|v_s|^{\frac{2n}{n-2}}\,\mathrm d\mu_{\hat g^s}\right)^{\frac{n-2}{2n}}.
\end{split}
\end{equation*}
From the nonnegativity of $R(\hat g^s)$ outside $\{2s\leq r\leq 3s\}$, we see
\begin{equation}\label{Eq: int scalar 1}
\left|\int_{\{s\leq r\leq 4s\}}\eta_sR(\hat g^s)\,\mathrm d\mu_{\hat g^s} \right|\leq \max\left\{\left|\int_{\{2s\leq r\leq 3s\}}R(\hat g^s)\,\mathrm d\mu_{\hat g^s}\right|,\left|\int_{\{s\leq r\leq 4s\}}R(\hat g^s)\,\mathrm d\mu_{\hat g^s}\right|\right\}.
\end{equation}
Recall from \cite[formula (4.2)]{Bartnik86} that
\[
R(\hat g^s)=|\hat g^s|^{-\frac{1}{2}}\partial_i\left(|\hat g^s|^{\frac{1}{2}}\hat g^{s,ij}\left(\hat\Gamma^s_j-\frac{1}{2}\partial_j(\log|\hat g^s|)\right)\right)-\frac{1}{2}\hat g^{s,ij}\hat\Gamma^s_i\partial_j\left(\log|\hat g^s|\right)+\hat g^{s,ij}\hat g^{s,kl}\hat g^{s,pq}\hat\Gamma^s_{ikp}\hat\Gamma^s_{jql},
\]
where
$$
\hat\Gamma^s_{ijk}=\frac{1}{2}\left(\hat g^s_{jk,i}+\hat g^s_{ik,j}-\hat g^s_{ij,k}\right)\quad \text{and}\quad\hat \Gamma^s_{k}=\hat g^{s,ij}\hat\Gamma^s_{ijk}.
$$
It follows from the fact $|\partial \hat g^s|=O(r^{1-n})$ that the last two terms in the expression of $R(\hat g^s)$ are $O(r^{2-2n})$. Here and in the sequel the constants involved in $O(\cdot)$ are independent of $s$. Let $(\lambda,\mu)=(1,4)$ or $(2,3)$. Then we have
\[
\begin{split}
\int_{\{\lambda s\leq r\leq \mu s\}}R(\hat g^s)\,\mathrm d\mu_{\hat g^s}&=\int_{\{\lambda s\leq r\leq \mu s\}}\partial_i\left(|\hat g^s|^{\frac{1}{2}}\hat g^{s,ij}\left(\hat\Gamma^s_j-\frac{1}{2}\partial_j(\log|\hat g^s|)\right)\right)\,\mathrm dx+O(s^{2-n})\\
&=\int_{\{r=\mu s\}}|\hat g^s|^{\frac{1}{2}}\hat g^{s,ij}\left(\hat\Gamma^s_j-\frac{1}{2}\partial_j(\log|\hat g^s|)\right)\,\mathrm d\sigma^i\\
&\qquad-\int_{\{r=\lambda s\}}|\hat g^s|^{\frac{1}{2}}\hat g^{s,ij}\left(\hat\Gamma^s_j-\frac{1}{2}\partial_j(\log|\hat g^s|)\right)\,\mathrm d\sigma^i+O(s^{2-n}),
\end{split}
\]
where $\mathrm d\sigma^i$ is the normal area element with respect to the Euclidean metric. Recall that we have $|\hat g^s_{ij}-\delta_{ij}|=O(r^{2-n})$ and $|\partial\hat g^s|=O(r^{1-n})$. So it holds
$$
|\hat g^s|^{\frac{1}{2}}\hat g^{s,ij}\left(\hat\Gamma^s_j-\frac{1}{2}\partial_j(\log|\hat g^s|)\right)=\hat g^s_{ij,j}-\hat g^s_{jj,i}+O(r^{3-2n}).
$$
Then we know
\begin{equation}\label{Eq: int scalar 2}
\int_{\{\lambda s\leq r\leq \mu s\}}R(\hat g^s)\,\mathrm d\mu_{\hat g^s}=\int_{\{r=\mu s\}}(\hat g^s_{ij,j}-\hat g^s_{jj,i})\,\mathrm d\sigma^i-\int_{\{r=\lambda s\}}(\hat g^s_{ij,j}-\hat g^s_{jj,i})\,\mathrm d\sigma^i+O(s^{2-n}).
\end{equation}
Combining \eqref{Eq: int scalar 1} and \eqref{Eq: int scalar 2}, we conclude that the second term on the right hand side of \eqref{Eq: 001} is no greater than
$$
\left|\int_{\{r=4s\}}(\hat g^s_{ij,j}-\hat g^s_{jj,i})\,\mathrm d\sigma^i-\int_{\{r=s\}}(\hat g^s_{ij,j}-\hat g^s_{jj,i})\,\mathrm d\sigma^i\right|+Cs^{-1}
$$
or
$$
\left|\int_{\{r=3s\}}(\hat g^s_{ij,j}-\hat g^s_{jj,i})\,\mathrm d\sigma^i-\int_{\{r=2s\}}(\hat g^s_{ij,j}-\hat g^s_{jj,i})\,\mathrm d\sigma^i\right|+Cs^{-1}.
$$
Combined with \eqref{Eq: no extra mass} and \eqref{Eq: expression gs}, it follows
\begin{equation}\label{Eq: scalar estimate}
\left|\int_{\{s\leq r\leq 4s\}}\eta_sR(\hat g^s)\,\mathrm d\mu_{\hat g^s}\right|=o(1),\quad \text{as}\quad s\to 0.
\end{equation}
Recall from \eqref{Eq: integral v} that
$$
\left(\int_U|v_s|^{\frac{2n}{n-2}}\,\mathrm d\mu_{\hat g^s}\right)^{\frac{n-2}{2n}}\leq 2c_S^{-1}\left(\int_{U}|\eta_s R(\hat g^s)-\delta_s\eta_s|^{\frac{2n}{n+2}}\,\mathrm d\mu_{\hat g^s}\right)^{\frac{n+2}{2n}}.
$$
Combined with \eqref{Eq: integral scalar} and \eqref{Eq: delta_s}, it yields
\begin{equation}\label{Eq: vs}
\left(\int_U|v_s|^{\frac{2n}{n-2}}\,\mathrm d\mu_{\hat g^s}\right)^{\frac{n-2}{2n}}=o(1),\quad \text{as}\quad s\to 0,
\end{equation}
and so we have
$$
\left|\int_U\eta_sR(\hat g^s) u_s\,\mathrm d\mu_{\hat g^s}\right|+\left|\int_U\delta_s\eta_s u_s\,\mathrm d\mu_{\hat g_s}\right|=o(1),\quad \text{as}\quad s\to 0.
$$
By taking $s$ large enough, we can guarantee $|A_s|\leq \epsilon/2$. Fix such an $s$ below.

Take
$$
u_{s,\tau}=\frac{u_s+\tau}{1+\tau}
$$
with $\tau$ a positive constant to be determined later. We consider the conformal deformation
$
\bar g=(u_{s,\tau})^{\frac{4}{n-2}}\hat g^s.
$
A straightforward computation gives
\begin{equation*}
\begin{split}
R(\bar g)&=(u_{s,\tau})^{-\frac{n+2}{n-2}}\left(-\frac{4(n-1)}{n-2}\Delta_{\hat g^s}u_{s,\tau}+R(\hat g^s)u_{s,\tau}\right)\\
&=(u_{s,\tau})^{-\frac{n+2}{n-2}}(1+\tau)^{-1}\left(\left((1-\eta_s)R(\hat g^s)+\delta_s\eta_s\right)u_s+R(\hat g^s)\tau\right).
\end{split}
\end{equation*}
Note that $R(\hat g^s)$ takes possible negative values only in $\{2s\leq r\leq 3s\}$, where the term $\left((1-\eta_s)R(\hat g^s)+\delta_s\eta_s\right)u_s$ has a positive lower bound. So we can take $\tau$ small enough such that $R(\bar g)$ is nonnegative everywhere. Since the function $u_{s,\tau}$ is no less than $\tau(1+\tau)^{-1}$, the metric $\bar g$ is still complete. In the region $\{r\geq 3s\}$, the metric $\bar g$ can be expressed as
$$
\bar g_{ij}=\left(u_{s,\tau}\left(1+\frac{m}{2r^{n-2}}\right)\right)^{\frac{4}{n-2}}\delta_{ij}=\left(1+\frac{\bar m}{2r^{n-2}}\right)^{\frac{4}{n-2}}+\bar h_{ij},
$$
where
$
|\bar m-m|=2(1+\tau)^{-1}|A_s|\leq \epsilon.
$

Now let us estimate $\|\bar g-g\|_g$ in the distinguished end $\mathcal E$. From the expression \eqref{Eq: expression gs} we have $\|\hat g^s-g\|_g=O(s^{2-n})$. To obtain $\|\bar g^s-g\|_g\leq \epsilon$ for sufficiently large $s$, all we need to show is $\|v_s\|_{L^\infty(\mathcal E)}\to 0$ as $s\to+\infty$. From the equation of $u_s$ and the relation $v_s=u_s-1$, we see
$$
\Delta_{\hat g^s} v_s-\frac{n-2}{4(n-1)}\left(\eta_s R(\hat g^s)-\delta_s\eta_s\right)v_s=\frac{n-2}{4(n-1)}\left(\eta_s R(\hat g^s)-\delta_s\eta_s\right).
$$
Clearly we have $\|\eta_s R(\hat g^s)-\delta_s\eta_s\|_{L^\infty(\mathcal E)}=O(s^{2-n})$. It follows from \cite[Theorem 8.17]{GT2001} that $\|v_s\|_{L^\infty(\mathcal E)}\to 0$ as $s\to+\infty$.
This completes the proof.
\end{proof}

Based on above proposition we have the following
\begin{theorem}\label{Thm: relate to torus}
If Conjecture \ref{Conj: Geroch} holds, then $(M,g,\mathcal E)$ has nonnegative ADM mass once it has nonnegative scalar curvature.
\end{theorem}
\begin{proof}
Let us follow Lohkamp's argument from \cite{Lohkamp99}. Suppose that $(M,g,\mathcal E)$ has negative ADM mass $m$. Set $\epsilon=-m/2$. From Proposition \ref{Prop: deformation} above we can construct a new complete metric $\bar g$ on $M$ such that $(M,\bar g,\mathcal E)$ is an asymptotically Schwarzschild manifold with arbitrary ends, which has nonnegative scalar curvature and ADM mass $\bar m$ no greater than $m/2$. Moreover, the metric $\bar g$ can be expressed as
$
\bar g_{ij}=u^{\frac{4}{n-2}}\delta_{ij}
$
around the infinity, which is also scalar flat. In particular, the function $u$ is a harmonic function around infinity with respect to the Euclidean metric. Notice that $u$ has the expansion
$$
u=1+\frac{\bar m}{2r^{n-2}}+O(r^{1-n}),\quad \bar m<0,
$$
so we can take $s_1$ large enough such that $u<1$ on $\{r=s_1\}$. Denote
$$
\epsilon=1-\sup_{r(x)=s_1}u(x).
$$
It is clear that $u>1-\epsilon/4$ in $\{r\geq s_2\}$ for sufficiently large $s_2>s_1$. Take a cutoff function $\zeta:[0,+\infty)\to [0,1-\epsilon/2]$ such that $\zeta(t)=t$ when $t\leq 1-3\epsilon/4$ and $\zeta(t)=1-\epsilon/2$ when $t\geq 1-\epsilon/4$. Moreover, we can also require $\zeta'\geq 0$ and $\zeta''\leq 0$ in $[0,+\infty)$ as well as $\zeta''<0$ in $(1-3\epsilon/4,1-\epsilon/4)$. Let
\begin{equation*}
\begin{split}
v=\left\{
\begin{array}{cc}
\zeta\circ u,&r\geq s_1;\\
u& r\leq s_1.
\end{array}\right.
\end{split}
\end{equation*}
Clearly $v$ is a smooth function defined on entire $M$ since $v$ equals to $u$ around $\{r=s_1\}$. A direct computation shows
$$
\Delta v=\zeta''|\nabla u|^2+\zeta'\Delta u\leq 0\quad \text{in}\quad \{r\geq s_1\}
$$
and further $\Delta v<0$ at some point in $\{s_1<r<s_2\}$, where $\Delta$ and $\nabla$ are Laplace and gradient operators with respect to the Euclidean metric. Define
$$
\tilde g=\left(\frac{v}{u}\right)^{\frac{4}{n-2}}\bar g.
$$
It is easy to verify that $\tilde g$ is a complete metric on $M$ with nonnegative scalar curvature (positive somewhere), which is exactly the Euclidean metric around the infinity of $\mathcal E$.

Take a large Euclidean cube $\mathcal C$ and we denote $M_{out}=\mathcal E-\Phi^{-1}(\mathcal C)$ and $M_{in}=M-M_{out}$, where $\Phi:\mathcal E\to \mathbb R^n-\bar B_1$ is the diffeomorphism from Definition \ref{Def 1}.
After identifying the opposite faces of cube $\mathcal C$, the quotient space $M_{in}/\sim$ is a smooth manifold in the form of $T^n\sharp X^n$ for some $n$-manifold $X$.  The metric $\tilde g$ induces a smooth complete metric on $M_{in}/\sim$ with nonnegative scalar curvature (positive somewhere). This leads to a contradiction to our assumption.
\end{proof}

Now, we are ready to prove the main theorem.

\begin{proof}[{P\bf{roof for Theorem \ref{Thm: main 1}}}]
Since Conjecture \ref{Conj: Geroch} was verified by Chodosh and Li \cite{CL2020} up to dimension $7$, we conclude that $(M,g,\mathcal E)$ has nonnegative ADM mass when the dimension $n$ is no greater than $7$.

So we just need to focus on the rigidity part in the following. First we show that the scalar curvature of $g$ has to vanish everywhere if the ADM mass of $(M,g,\mathcal E)$ is zero. Otherwise, we can take a neighborhood $U$ of $\mathcal E$ such that $R(g)$ is positive at some point $p$ in $U$. Take a nonnegative cutoff function $\eta:M\to [0,1]$ with compact support in $U$ such that $\eta\equiv 1$ around point $p$. It follows from Proposition \ref{Prop: the conformal factor} that there is a positive function $u$ solving
$$
\Delta_g u-\frac{n-2}{4(n-1)}\eta R(g)u=0,
$$
which has the expansion $u=1+Ar^{2-n}+O(r^{1-n})$ with
$$
A=-\frac{1}{4(n-1)|\mathbb S^{n-1}|}\int_U\eta R(g)u\,\mathrm d\mu_g<0.
$$
Define
$$
\bar g=\left(\frac{u+1}{2}\right)^{\frac{4}{n-2}}g.
$$
We emphasize that the choice of the conformal factor $(u+1)/2$ is to guarantee the completeness of metric $\bar g$ (if the conformal factor tends to zero at infinity of those bad ends, then metric $\bar g$ may be incomplete). It is not difficult to verify that $(M,\bar g,\mathcal E)$ is an asymptotically flat manifold with arbitrary ends, which has nonnegative scalar curvature but negative ADM mass. This is impossible from our previous discussion.

Next we prove the Ricci flatness of the metric $g$. Otherwise we take a neighborhood $U$ of $\mathcal E$ such that $\Ric(g)$ does not vanish at some point $p$ in $U$. Take a nonnegative cutoff function $\eta:M\to [0,1]$ with compact support $S$ in $U$ such that $\eta\equiv 1$ around point $p$. After applying a variation argument to the first Neumann eigenvalue of conformal Laplace operator (see \cite[proof of Lemma 3.3]{Kazdan82}), we can pick up a positive contant $\epsilon$ small enough such that the metric $\bar g=g-\epsilon\eta\Ric(g)$ satisfies
\begin{equation}\label{Eq: 002}
\int_{S}|\nabla_{\bar g}\zeta|^2+\frac{n-2}{4(n-1)}R(\bar g)\zeta^2\,\mathrm d\mu_{\bar g}>c\int_{S}\zeta^2\,\mathrm d\mu_{\bar g},\quad\forall\,\zeta\neq 0\in C^\infty(S),
\end{equation}
for some positive constant $c$. Denote $c_S$ to be the Sobolev constant of $U$ with respect to the metric $\bar g$. Since $\bar g$ is almost equal to $g$, we can also require
\begin{equation*}
\left(\int_{U}|(R(\bar g))_-|^{\frac{n}{2}}\,\mathrm d\mu_{\bar g}\right)^{\frac{2}{n}}\leq \frac{c_S}{4}
\end{equation*}
by further decreasing the value of $\epsilon$. Let $\tilde\eta$ be another nonnegative cutoff function with compact support $\tilde S$ in $U$, which has a positive lower bound in the support of $\eta$. For $\delta$ small enough, we can construct a smooth positive function $u$ from Proposition \ref{Prop: the conformal factor} solving
$$
\Delta_{\bar g}u-\frac{n-2}{4(n-1)}(R(\bar g)-\delta\tilde\eta)u=0.
$$
The function $u$ has the expansion $u=1+Ar^{2-n}+O(r^{1-n})$ and we expect $A$ to be negative when $\delta$ is sufficiently small. From integration by parts we have
\begin{equation*}
\begin{split}
(n-2)|\mathbb S^{n-1}|A&=\int_{\partial U}u\frac{\partial u}{\partial \vec n}\,\mathrm d\sigma_{\bar g}-\int_{U}|\nabla_{\bar g}u|^2+\frac{n-2}{4(n-1)}\left(R(\bar g)-\delta\tilde \eta\right)u^2\,\mathrm d\mu_{\bar g}\\
&\leq -\int_{U}|\nabla_{\bar g}u|^2+\frac{n-2}{4(n-1)}R(\bar g)u^2\,\mathrm d\mu_{\bar g}+\delta\int_{\tilde S}u^2\,\mathrm d\mu_{\bar g}.
\end{split}
\end{equation*}
From \eqref{Eq: 002} and the fact $R(\bar g)\equiv 0$ outside $\tilde S$, it follows
$$
\int_{U}|\nabla_{\bar g}\zeta|^2+\frac{n-2}{4(n-1)}R(\bar g)\zeta^2\,\mathrm d\mu_{\bar g}>\tilde c\int_{\tilde S}\zeta^2\,\mathrm d\mu_{\bar g},\quad\forall\,\zeta\neq 0\in C^\infty(U),
$$
for some positive constant $\tilde c$ independent of $\delta$. So we can guarantee $A<0$ by taking $\delta$ small enough. As before, we consider the conformal metric
$$
\tilde g=\left(\frac{u+\tau}{1+\tau}\right)^{\frac{4}{n-2}}\bar g.
$$
A straightforward computation shows
$$
R(\tilde g)=(1+\tau)^{\frac{4}{n-2}}(u+\tau)^{-\frac{n+2}{n-2}}\left(\delta\tilde\eta u+\tau R(\bar g)\right).
$$
In particular, we have $R(\tilde g)\geq 0$ if $\tau$ is chosen to be small enough. However, it is easy to verify that $(M,\tilde g,\mathcal E)$ is an asymptotically flat manifold with arbitrary ends whose ADM mass is negative. Again this contradicts to our previous discussion.

Now we can deduce that $M$ has only one end $\mathcal E$. Otherwise, the Cheeger-Gromoll splitting theorem (see \cite{CG1971}) yields that $(M,g)$ must be isometric to a Riemannian product manifold $N\times \mathbb R$ for a closed manifold $N$. In particular, each $N$-slice is totally geodesic. However, this is impossible since such slice cannot appear at the infinity of the asymptotically flat end $\mathcal E$. Finally, the Bishop-Gromov volume comparison theorem implies that $(M,g)$ is isometric to the Euclidean space.
\end{proof}

Now we can derive Corollary \ref{Cor: ALE} from Theorem \ref{Thm: main 1} directly by a simple lifting argument.
\begin{proof}[P\bf{roof of Corollary \ref{Cor: ALE}}]
Let us consider the universal covering
$
p:(M,g)\to (M_\Gamma,g_{\Gamma}).
$
Clearly $(M,g)$ has nonnegative scalar curvature as well. Let $\mathcal E$ be one of the connected components of $p^{-1}(\mathcal E_\Gamma)$ and it is easy to verify that $\mathcal E$ is still a covering of $\mathcal E_\Gamma$.
Since the inclusion map $i_*:\pi_1(\mathcal E_\Gamma)\to \pi_1(M_\Gamma)$ is injective, we claim that $\mathcal E$ has to be simply connected and so it is diffeomorphic to $\mathbb R^n-\bar B_1$. Otherwise, there is a homotopically non-trivial loop $\gamma$ in $\mathcal E$ that shrinks to a point in $M$. Correspondingly, the image $p(\gamma)$ is a loop in $\mathcal E_\Gamma$ that is homotopic to a point in $M_\Gamma$. Now the injectivity of $i_*$ implies that $p(\gamma)$ shrinks to a point in $\mathcal E_\Gamma$. We obtain a contradiction by simply lifting this homotopy to $\mathcal E$. From definition we see that $(M,g,\mathcal E)$ is an asymptotically flat manifold with arbitrary ends whose ADM mass equals to $|\Gamma|\cdot m$, where $m$ is the ADM mass of $(M_\Gamma,g_{\Gamma},\mathcal E_\Gamma)$. It follows from Theorem \ref{Thm: main 1} that $m\geq 0$. It rests to rule out the possibility that $m=0$. In this case, $(M,g)$ is isometric to the Euclidean $n$-space. From the theory of covering spaces the Deck transformation group of $\mathbb R^n$ here is isometric to the fundamental group $\pi_1(M_\Gamma)$, which contains a non-trivial finite group $G=i_*(\pi_1(\mathcal E_\Gamma))$ due to our assumption. Since each element $g$ in $G$ has the form
$$
g:\mathbb R^n\to \mathbb R^n,\quad x\mapsto T_g\cdot x+v_g,\quad T_g\in O(n),\quad v_g\in\mathbb R^n,
$$
the group $G$ has a fixed point
$$
\frac{1}{|G|}\sum_{g\in G}g(0).
$$
However, only identity can have fixed points among all Deck transformations and this leads to a contradiction.
\end{proof}

\bibliography{bib}
\bibliographystyle{amsplain}
\end{document}